\documentclass[11pt]{amsart}
\usepackage{amssymb}
\usepackage[dvips]{graphicx}
\usepackage{bbm}
\usepackage{url}
\usepackage{color}
\definecolor{darkblue}{rgb}{0.15,.4,.5}
\usepackage[naturalnames,
            colorlinks,linkcolor=darkblue,
            citecolor=darkblue,pagecolor=darkblue,
            filecolor=darkblue,urlcolor=darkblue,
            anchorcolor=darkblue,menucolor=darkblue,
            bookmarksopen=false,bookmarks=false,
            ps2pdf=true
           ]{hyperref}
\usepackage{breakurl}

\numberwithin{equation}{section}

\newtheorem{theorem}{Theorem}
\newtheorem{proposition}[theorem]{Proposition}
\newtheorem{lemma}[theorem]{Lemma}

\theoremstyle{definition}

\newtheorem{remark}[theorem]{Remark}
\newtheorem{question}[theorem]{Question}

\newtheorem{conj}[theorem]{Conjecture}

\numberwithin{theorem}{section}

\newcommand{\pro}[2]{\langle #1, #2 \rangle}

\newcommand{\W}{\mathcal{W}}

\newcommand{\Z}{\mathbb{Z}}
\newcommand{\Q}{\mathbb{Q}}
\newcommand{\R}{\mathbb{R}}

\renewcommand{\P}{\mathbb{P}}

\newcommand{\V}{\mathcal{V}}

\def\ca{{\rm ca}}
\def\Pic{{\rm Pic}}
\def\MR{M_\R}
\def\NR{N_\R}
\def\vol{{\rm vol}}

\def\lct{{\rm lct}}
\def\Aut{{\rm Aut}}

\DeclareMathOperator{\aff}{aff}

\title[Non-symmetric K\"ahler-Einstein toric Fano manifolds]
{Examples of K\"ahler-Einstein toric Fano manifolds associated to non-symmetric reflexive polytopes}

\author{Benjamin Nill}
\address{Department of Mathematics, University of Georgia, Athens, GA, 30602, USA}
\email{bnill@math.uga.edu}

\author{Andreas Paffenholz}
\address{Fachbereich Mathematik, TU Darmstadt, 64289 Darmstadt, Germany}
\email{paffenholz@math.fu-berlin.de}



\begin{document}

\begin{abstract}
In this note we report on examples of $7$- and $8$-dimen\-sional toric Fano manifolds 
whose associated reflexive polytopes are not symmetric, but they still admit a K\"ahler-Einstein metric.
This answers a question first posed by V.V. Batyrev 
and E. Selivanova. The examples were found in the 
classification of $\leq 8$-dimensional toric Fano manifolds obtained by M. {\O}bro. 
We also discuss related open questions and conjectures. In particular, we notice that 
the alpha-invariants of these examples do not satisfy the assumptions of Tian's theorem.
\end{abstract}

\maketitle

\section{Introduction}

Let us first fix our setting. 
In the toric case, there is a correspondence between 
$n$-dimensional nonsingular Fano varieties and $n$-dimensional Fano polytopes, where the Fano varieties 
are biregular isomorphic if and only if the corresponding Fano polytopes are unimodularly equivalent. 
Here, given a lattice $N$ of rank $n$, a {\em Fano polytope} $Q \subseteq \NR := N \otimes_\Z \R$ 
is given as a lattice polytope containing the origin strictly in its interior such that 
the vertices of any facet of $Q$ form a lattice basis of $M$. In this case, when we denote the dual lattice by $M$, 
the dual polytope is given as 
\[P := Q^* := \{y \in \MR \,:\, \pro{y}{x} \geq -1 \;\forall\, x \in Q\}.\]
Since $Q$ is a Fano polytope, $P$ is also a lattice polytope. In particular, $Q$ and $P$ are {\em reflexive} polytopes 
in the sense of V.V.\,Batyrev \cite{Bat94}.

In 2003 X. Wang and X. Zhu clarified completely which nonsingular toric Fano varieties 
admit a K\"ahler--Einstein metric \cite{WZ04}:

\begin{theorem}[Wang/Zhu] Let $X$ be a nonsingular toric Fano variety with associated reflexive polytope $P$. 
Then $X$ admits a K\"ahler--Einstein metric if and only if the barycenter $b_P$ of $P$ is zero.
\end{theorem}

Here, the barycenter of $P$ equals the Futaki character of the holomorphic vector field of $X$~\cite{Mab87}.
It is also known that the existence of a K\"ahler--Einstein metric implies that the automorphism group of $X$ is reductive. 
The converse does not hold (for related combinatorial questions see also \cite{Nil06}).

Prior to the previous theorem, in 1999 V.V. Batyrev and E. Selivanova had already proved a sufficient condition. 
For this, let us denote by $\W(P)$ the group of lattice automorphisms of $M$ that map $P$ onto itself. 
Now, $P$ is called {\em symmetric}, if the origin is the only lattice point of $M$ fixed by all elements of $\W(P)$. 
Note that $P$ is symmetric if and only if $Q$ is symmetric, see Lemma~\ref{aut}.

\begin{theorem}[Batyrev/Selivanova] Let $X$ be a nonsingular toric Fano variety with associated reflexive polytope $P$. 
If $P$ is symmetric, then $X$ admits a K\"ahler--Einstein metric.
\end{theorem}

\begin{question}[Batyrev/Selivanova] Does the converse also hold?
\label{ques}
\end{question}

This question was also posed by J. Song (remark after Proposition 4.3 of \cite{Son05}), by K. Chan and N.C. Leung (Remark 4.1 of \cite{CL07}), 
and by A. Futaki, H. Ono, and Y. Sano (introduction of version \texttt{v1} of \cite{FOS08} and Remark 1.4 of \cite{San08}). 
The hope was that several technical assumptions may be omitted, if the answer would be positive. 
Unfortunately, this is not true in higher dimensions.

\begin{proposition}The answer to Question~\ref{ques} is negative, if $n \geq 7$, 
and affirmative, if $n \leq 6$.
\end{proposition}

This observation is explained in the next section. Related open questions and conjectures are 
discussed in the last section of the paper.\smallskip

\textbf{Acknowledgment. }{\rm
We thank Ivan Cheltsov for his interest and helpful comments.
}

\section{The examples}

M.   {\O}bro  described  in  \cite{Oeb07} an  efficient  algorithm  to
classify Fano polytopes, that he used to compute complete lists of all
isomorphism  classes  of  $n$-dimensional  Fano polytopes  (and  their
duals)  for  $n  \leq  8$.   Now, an  exhaustive  computer  search  in
{\O}bro's database shows  that up to this dimension  there are exactly
three  examples  on  non-symmetric  smooth reflexive  polytopes  whose
barycenter is the origin.  The necessary computations were implemented
in    the    software   system    \texttt{polymake}~\cite{Joswig2009}.
\texttt{polymake}   relied  on   \cite{normaliz2,   nauty}  for   some
calculations.  To  explain the  examples, let us  denote by  $b_Q$ the
barycenter and by $v_Q$ the sum of all vertices of $Q$.

\begin{proposition} 
  Let $n  \leq 8$,  and $Q$ be  an $n$-dimensional Fano  polytope with
  dual polytope $P$ such that $b_P = 0$. Then $Q$ is symmetric, except
  if $Q$  is one of the  following three Fano polytopes  $Q_1, Q_2$ or
  $Q_3$:

  \begin{list}{({\theenumi})}{\usecounter{enumi}\leftmargin10pt\itemsep2pt}
  \item $Q_1$  is $7$-dimensional and  has $12$ vertices given  by the
    columns of the following matrix:

    \smallskip

    $\left(\begin{array}{rrrrrrrrrrrr}
        \phantom{-}1 & \phantom{-}0 & \phantom{-}0 & \phantom{-}0  &  \phantom{-}0 & -1 & \phantom{-}0 & \phantom{-}0 & \phantom{-}0 & \phantom{-}0  & \phantom{-}0 & \phantom{-}0\\
        0 & 1 & 0 & 0  & -1 & 0  & 0 & 0 & 0 & 0  & 0 & 0\\
        0 & 0 & 1 & -1 &  0 & 0  & 0 & 0 & 0 & 0  & 0 & 0\\
        0 & 0 & 0 & 0  &  0 & 0  & 1 & 0 & 0 & -1 & 0 & 0\\
        0 & 0 & 0 & 0  &  0 & 0  & 0 & 1 & 0 & -1 & 0 & 0\\
        0 & 0 & 0 & 0  &  0 & 0  & 0 & 0 & 1 & -1 & 0 & 0\\
        0 & 0 & 0 & -1 & -1 & -1 & 0 & 0 & 0 & 2 & 1 & -1
      \end{array}\right)$

    \smallskip

    The  associated  non-singular  toric   Fano  variety  $X_1$  is  a
    $\P^1$-bundle over $(\P^1)^3 \times \P^3$.

  \item $Q_2$ is the  $8$-dimensional Fano polytope with $14$ vertices
    corresponding  to $X_2  := X_1  \times \P^1$  (i.e., $Q_2$  is the
    bipyramid over $Q_1$).

  \item $Q_3$  is $8$-dimensional and  has $16$ vertices given  by the
    columns of the following matrix.

\smallskip

{\setlength{\arraycolsep}{3pt}
$\left(\begin{array}{rrrrrrrrrrrrrrrr}
1 & \phantom{-}0 & \phantom{-}0 & \phantom{-}0  &  \phantom{-}0 & -1 & \phantom{-}0 & \phantom{-}0 & \phantom{-}0 & \phantom{-}0  & \phantom{-}0 & \phantom{-}0  & \phantom{-}0  & \phantom{-}0  & \phantom{-}0  & \phantom{-}0\\
0 & 1 & 0 & 0  & -1 & 0  & 0 & 0 & 0 & 0  & 0 & 0  & 0  & 0  & 0  & 0\\
0 & 0 & 1 & -1 &  0 & 0  & 0 & 0 & 0 & 0  & 0 & 0  & 0  & 0  & 0  & 0\\
0 & 0 & 0 & 0  &  0 & 0  & 1 & 0 & 0 & -1 & 0 & 0  & 0  & 0  & 0  & 0\\
0 & 0 & 0 & 0  &  0 & 0  & 0 & 1 & 0 & -1 & 0 & 0  & 0  & 0  & 0  & 0\\
0 & 0 & 0 & 0  &  0 & 0  & 0 & 0 & 1 & -1 & 0 & 0  & 0  & 0  & 0  & 0\\
0 & 0 & 0 & 0  & 0  & 0  & 0 & 0 & 0 & 0  & 0 & 1  & -1 & 1  & -1 & 0\\
0 & 0 & 0 & -1 & -1 & -1 & 0 & 0 & 0 & 2  & 1 & 0  & 1  & -1 & 0  & -1
\end{array}\right)$}

\smallskip

The associated non-singular toric Fano variety $X_3$ is a $S_6$-bundle
over $(\P^1)^3 \times \P^3$, where $S_6$ is the del Pezzo surface with
$\rho_{S_6} =  4$, which is  $\P^2$ blown-up at  three torus-invariant
fix points.
  \end{list}

  All three examples have  a $1$-dimensional fixed space $F_{Q_i}$ for
  their automorphism group, generated by $v_{Q_i}$, for $i=1,2,3$.
\label{main}
\end{proposition}

In particular,  for each $n \geq  7$ we see  that $(\P^1)^{n-7} \times
X_1$   is   a   toric  Fano   $n$-fold   admitting   a 
K\"ahler--Einstein metric, while the associated reflexive polytope is not symmetric.

\begin{remark}{\rm 
Shortly after this paper had appeared as a preprint \cite{NP09}, 
H.\,Ono, Y.\,Sano, and N.\,Yotsutani proved that the Fano manifold $X_1$ 
is also the first known example of an asymptotically Chow unstable manifold with constant scalar curvature, 
see \cite{OSY09}.
}\end{remark}

\section{Related questions and results}

\subsection{The alpha-invariant and the log canonical threshold}

There is a celebrated theorem \cite{Tia87} which gives a sufficient condition for 
the existence of a K\"ahler-Einstein metric on a Fano $n$-fold in terms of the so-called 
{\em alpha-invariant}.

\begin{theorem}[Tian]
Let $X$ be a Fano $n$-fold and $G \subset \Aut(X)$ a compact
subgroup such that $\alpha_G(X) > \frac{n}{n+1}$. 
Then $X$ admits a K\"ahler-Einstein metric.
\label{tian}
\end{theorem}

In the case of a toric Fano $n$-fold $X$ there is an explicit formula \cite{Son05} 
for the alpha-invariant. For this, let $P$ be the associated reflexive polytope. 
We denote by $P_G$ the intersection of $P$ with the subspace $F_G$ consisting of all points that 
are fixed by all elements of the group $G \subseteq \W(P)$. Let us also recall the definition of the {\em coefficient of asymmetry} $\ca(P,0)$ of $P$ about the origin:
\[\ca(P,0) := \max_{\|y\| = 1} \frac{\max(\lambda > 0 \,:\, \lambda y \in P)}{\max(\lambda > 0 \,:\, -\lambda y \in P)}.\]
The coefficient of asymmetry plays also an important role in finding upper bounds on 
the volume of lattice polytopes with a fixed number of interior lattice points \cite{Pik01}. 

\begin{theorem}[Song]
Let $X$ be an $n$-dimensional toric Fano manifold with associated reflexive polytope $P$. Let $G$ be the subgroup of $\Aut(X)$ generated by $\W(P)$ and 
$(S^1)^n$. Then $\alpha_G(X) = 1$, if $P$ is symmetric, and 
$\alpha_G(X) = \frac{1}{1+\ca(P_{\W(P)},0)} \leq \frac{1}{2}$, otherwise.
\label{song}
\end{theorem}

In \cite{CS08a} it was shown that for $X$ smooth and $G$ compact, the alpha-invariant $\alpha_G(X)$ coincides with the {\em global $G$-invariant 
log canonical threshold }$\lct(X,G)$, for this notion see Definition~1.13 in \cite{CS08a}. 
I.\,Cheltsov and C.\,Shramov also calculated directly the log canonical threshold without assuming smoothness 
(see also Remark 1.11 in \cite{CS08b}):

\begin{lemma}[Cheltsov/Shramov]
Let $X$ be an $n$-dimensional toric $\Q$-factorial Fano variety with the polytope $P$ associated to $-K_X$. 
Let $G\subset\mathcal{W(P)}$ be a subgroup. Then
$$
\mathrm{lct}\Big(X,\ G\Big)=\frac{1}{1 + \mathrm{max}\Big\{\big\langle w, v\big\rangle\ \big\vert\ w\in P_G,\ v\in\V(Q)\Big\}},
$$
where $\V(Q)$ are the primitive generators of the fan associated to $X$.
\label{cheltsov}
\end{lemma}

From a combinatorial point of view, it is indeed straightforward to notice that the previous two formulas 
in Theorem \ref{song} and Lemma \ref{cheltsov} agree for a reflexive polytope $P$ dual to a Fano polytope $Q$. 
Now, let us compute the alpha-invariants of our examples:

\begin{proposition}
  Each Fano polytope  $Q_1,Q_2,Q_3$ (see Proposition~\ref{main}) has a
  $1$-dimensional    fixed space    under   the    action    of   $G    =
  \mathcal{W(P)}$.  Hence, this  also holds  for their  dual reflexive
  polytopes. Therefore, $\alpha_G(X_i) = \frac{1}{2}$ for $i=1,2,3$.
\end{proposition}

This follows immediately from the following well-known fact, whose 
proof we provide for the convenience of the reader (see Proposition 5.4.2 in \cite{Nil05}):

\begin{lemma}
Let $P,P^*$ be dual reflexive polytopes. Then the fixed spaces of the automorphism groups 
of $P,P^*$ have the same dimension.
\label{aut}
\end{lemma}

\begin{proof}
In fact, Maschke's theorem yields $\MR \cong F_{\W(P)} \oplus U$ for some $\W(P)$-invariant 
subspace $U \subseteq \MR$. Dualizing yields $\NR \cong (F_{\W(P)})^* \oplus U^*$, 
hence, $\dim_\R F_{\W(P)} = \dim_\R (F_{\W(P)})^* \leq \dim_\R F_{\W(P)^*}$. Now, 
the statement follows by symmetry and by observing that $\W(P)^* = \W(P^*)$.
\end{proof}

Note that in the toric case, $\alpha_G(X) > \frac{n}{n+1}$ (rather $\alpha_G(X) > \frac{1}{2}$) is actually equivalent to 
$P_G = \{0\}$ for $G\subset\mathcal{W(P)}$ by Lemma~\ref{cheltsov}, respectively, Theorem~\ref{song}. 
In particular, for any of these three examples $X$ the converse of Theorem \ref{tian} does not hold 
for {\em any} compact subgroup $G \subset \Aut(X)$. 

\subsection{Chern number inequalities}

In \cite{CL07} a series of Miyaoka-Yau type inequalities were proposed by K. Chan and N.C. Leung for 
compact K\"ahler $n$-folds $X$ with negative $c_1(X)$. In the toric case they also conjectured 
an analogue for positive $c_1(X)$.

\begin{conj}[Chan/Leung]{\rm 
Let $X$ be a K\"ahler-Einstein toric Fano $n$-fold. Then 
\[c_1^2(X) H^{n-2} \leq 3 c_2(X) H^{n-2}\]
for any nef class $H$.
\label{clconj}
}
\end{conj}

\begin{remark}  Here  is  a purely  combinatorial consequence  that  was
  observed in \cite{CL07}.  Let $Q$ be the Fano polytope corresponding
  to a  K\"ahler-Einstein toric  Fano $n$-fold $X$,  and $P$  the dual
  reflexive polytope. Let us  denote the Ehrhart polynomial $k \mapsto
  |(k P) \cap \Z^n|$ of $P$ by $\sum_{i=0}^n a_i t^i$. Then
\[c_1^2(X) (-K_X)^{n-2} \leq 3 c_2(X) (-K_X)^{n-2}\]
if and only if
\begin{equation}
a_{n-2} \leq \frac{1}{3} \vol(P^{(2)}),
\label{eq}
\end{equation}
where  $P^{(2)}$ is the  union of  all codimension  two faces  of $P$.
Using the  database, we checked that Equation~(\ref{eq})  holds for $n
\leq   7$.    This  provides   additional   evidence   in  favour   of
Conjecture~\ref{clconj}.
\end{remark}

In their paper K. Chan and N.C. Leung proved this conjecture in some particular instances (Theorem 1.1 of \cite{CL07}):

\begin{theorem}[Chan/Leung]
Conjecture~\ref{clconj} holds, if 
\begin{enumerate}
\item $n=2,3,4$, or 
\item each facet of the associated reflexive polytope $P$ contains a lattice point in its interior.
\end{enumerate}
\label{above}
\end{theorem}

The proof relied on a purely combinatorial property, which the authors conjectured to 
hold also without additional assumptions on $X$ (Conjecture 3.1 of \cite{CL07}):

\begin{conj}[Chan/Leung]{\rm
If $b_P=0$, then for any facet $F$ there exists a point $x_F \in \aff(F)$ such that
\[\pro{u_G}{x_F} \leq \frac{1}{2}\]
for any facet $G$ of $P$ adjacent to $F$, which is defined via $\pro{u_G}{G} = -1$ and 
$\pro{u_G}{P} \geq -1$.
\label{yau}
}
\end{conj}

Here is an example that shows that this approach of proving Conjecture~\ref{clconj} 
unfortunately fails in general.

\begin{proposition}
 Conjecture~\ref{yau} does not hold for the $5$-dimensional reflexive
 polytope $P$  with  $b_P = 0$,  whose dual  $Q$ is a   Fano polytope
 having the following vertices:

\smallskip

$\left(\begin{array}{rrrrrrrr}
-1&  0&  0&  0&  \phantom{-}0&  0&  \phantom{-}1&  0\\
0& -1&  0&  0&  0&  0&  0&  1\\
0&  0& -1&  0&  0&  0&  1&  0\\
0&  0&  0& -1&  1&  0&  2& -2\\
0&  0&  0&  0&  0& -1&  0&  1
\end{array}\right)$
\end{proposition}

\smallskip

The criterion of Conjecture~\ref{yau} does  not hold for the facets of
$P$ associated to  the vertices in columns   four and five.   Note that
these are  precisely the facets  that do not  contain interior lattice
points, as required by Theorem~\ref{above}.

\subsection{The anticanonical degree}

There is  a long-standing open conjecture  by E.~Ehrhart \cite{Ehr55},
see   also  \cite[Section  E 13]{CFG91},   that  can   be  seen   as  a
generalization of Minkowski's first theorem:

\begin{conj}[Ehrhart] Let $P \subseteq \MR$ be an $n$-dimensional convex body 
with the origin as its only interior lattice point and barycenter $b_P = 0$. 
Then $\vol(P) \leq (n+1)^n / n!$. 
Moreover, equality should only be obtained for $Q^*$, 
where $Q$ is the (unique) Fano simplex corresponding to $\P^n$.
\label{ehrhart}
\end{conj}

It was shown by E.~Ehrhart that the conjecture holds for $n=2$. 
We also checked it for duals of Fano polytopes up to dimension eight. 
The best upper bound on $\vol(P)$ in the situation of Conjecture~\ref{ehrhart} is 
$\vol(P) \leq (n+1)^n (1-((n-1)/n)^n) \leq (n+1)^n$. A survey on this and related results 
can be found in \cite{GW93}.

In algebro-geometric terms, Conjecture~\ref{ehrhart} implies the following statement, for which no proof is 
known, too: 
Any $n$-dimensional toric Fano manifold $X$ that admits a K\"ahler--Einstein metric 
has anticanonical degree $(-K_X)^n \leq (n+1)^n$, with equality only for $\P^n$. 
We recall that $(-K_X)^n = n!\, \vol(P)$, where $P$ is the reflexive polytope that is dual to the Fano polytope 
associated to $X$.
Now, it was noted in \cite{GMSY07} that Bishop's obstruction \cite{BC64} yields the 
following bound:
\begin{equation}I(X) (-K_X)^n \leq (n+1)^{n+1},
\label{bound}
\end{equation}
where $I(X)$ is the {\em Fano index}, i.e., the largest positive integer dividing $K_X$ in $\Pic(X)$. 
While this inequality seems slightly weaker than 
Conjecture~\ref{ehrhart}, note that it is sharp for $\P^n$, since $I(\P^n) = n+1$. 
Here, the Fano index 
$I(X)$ can be also computed as the largest positive integer $i$ such that $(P-v)/i$ is still a lattice polytope, 
for $P$ given as above and $v$ some vertex of $P$. However, we don't know of a purely convex-geometric 
proof of (\ref{bound}), 
which we pose as a challenge to the combinatorial community. 

\begin{remark} 
  For a  general toric Fano $n$-fold  $X$ there is  {\em no polynomial
    bound} on  $\sqrt[n]{(-K_X)^n}$, as  was proved by  O.\,Debarre in
  \cite[p.\ 139]{Deb01}.  In particular, also inequality~(\ref{bound})
  does  not hold  in general,  as had  been suggested  in  some recent
  papers (see  \cite[Conj.\ 6.4]{Spa08}, \cite[Conj.\  1.8]{CS09}, and
  \cite[inequality (2.22)]{GMSY07}).
\end{remark}

\end{document}